\documentclass[reqno]{amsart}
\usepackage{amsmath,amsthm,amssymb,mathrsfs,stmaryrd,eucal,mathbbol}
\usepackage[all,cmtip]{xy}
\usepackage{hyperref}

\numberwithin{equation}{section}

\newcommand{\Z}{\mathbb{Z}}

\newcommand{\CC}{\mathcal{C}}
\newcommand{\CD}{\mathcal{D}}
\newcommand{\CE}{\mathcal{E}}

\newcommand{\CM}{\mathcal{M}}
\newcommand{\CN}{\mathcal{N}}

\newcommand{\FZ}{\mathfrak{Z}}

 \DeclareMathOperator{\Hom}{Hom}

 \DeclareMathOperator{\Id}{Id}

 \DeclareMathOperator{\Fun}{Fun}

 \DeclareMathOperator{\LMod}{LMod}
 \DeclareMathOperator{\RMod}{RMod}
 \DeclareMathOperator{\BMod}{BMod}

\newcommand{\bk}{{\mathbf{k}}}
\newcommand{\one}{{\mathbf{1}}}
\newcommand{\op}{{\mathrm{op}}}
\newcommand{\rev}{{\mathrm{rev}}}
\newcommand\void[1]{}

\renewcommand{\odot}{\otimes}

\newtheorem{thm}{Theorem}[section]
\newtheorem{lem}[thm]{Lemma}
\newtheorem{prop}[thm]{Proposition}
\newtheorem{cor}[thm]{Corollary}

\newtheorem{prop-defn}[thm]{Proposition-Definition}

\theoremstyle{definition}
\newtheorem{defn}[thm]{Definition}

\newtheorem{exam}[thm]{Example}
\newtheorem{rem}[thm]{Remark}

\theoremstyle{remark}

\begin{document}

\title{Semisimple and separable algebras in multi-fusion categories}
\author{Liang Kong}
\address{Shenzhen Institute for Quantum Science and Engineering, and Department of Physics, Southern University of Science and Technology, Shenzhen 518055, China }
\email{kongl@sustech.edu.cn}
\author{Hao Zheng}
\address{Shenzhen Institute for Quantum Science and Engineering, and Department of Physics, Southern University of Science and Technology, Shenzhen 518055, China. \newline
\indent Department of Mathematics, Peking University, Beijing 100871, China}
\email{zhengh@sustech.edu.cn}
\maketitle

\begin{abstract}
We give a classification of semisimple and separable algebras in a multi-fusion category over an arbitrary field in analogy to Wedderben-Artin theorem in classical algebras. It turns out that, if the multi-fusion category admits a semisimple Drinfeld center, the only obstruction to the separability of a semisimple algebra arises from inseparable field extensions as in classical algebras. Among others, we show that a division algebra is separable if and only if it has a nonvanishing dimension.
\end{abstract}

\section{Introduction}

Fusion categories and their generalization, multi-fusion categories, have attracted a lot of attentions recently not only because of their beautiful theory (see \cite{ENO,EGNO} and references therein) but also because of their important applications in other areas such as topological field theory and condensed matter physics.

In this paper, we give a systematic study of two classes of very basic but very rich algebras in a multi-fusion category: semisimple algebras and separable algebras. First, we give a classification of semisimple algebras in terms of division algebras (Theorem \ref{thm:semisimple}) in the spirit of Wedderben-Artin theorem in classical algebras, as well as a classification of separable algebras (Theorem \ref{thm:separable}) together with several separability criteria.

By definition, a multi-fusion category is a rigid semisimple monoidal category. If we assume further that the multi-fusion category admits a semisimple Drinfeld center, then the theory becomes more consistent with classical algebras. It turns out that the only obstruction to the separability of a semisimple algebra arises from inseparable field extensions (Theorem \ref{thm:cen-nonper-sep}, Corollary \ref{cor:cen-per-sep}, Theorem \ref{thm:zero-sep}).


Among others, we introduce the notion of the dimension of a division algebra (Definition \ref{def:quan-dim}), and show that a division algebra is separable if and only if it has a nonvanishing dimension (Theorem \ref{thm:div-sep-dim}). 

\medskip

\noindent {\bf Acknowledgement.} 
We are supported by the Science, Technology and Innovation Commission of Shenzhen Municipality (Grant Nos. ZDSYS20170303165926217 and JCYJ20170412152620376) and Guangdong Innovative and Entrepreneurial Research Team Program (Grant No. 2016ZT06D348), and by NSFC under Grant No. 11071134. LK is also supported by NSFC under Grant No. 11971219. HZ is also supported by NSFC under Grant No. 11131008 and 11871078.


\section{Multi-fusion categories}

In this section, we recall some basic facts about monoidal categories and multi-fusion categories. We refer readers to the book \cite{EGNO} for a general reference. We will follow the notations in \cite{KZ}.

\medskip

Given left modules $\CM,\CN$ over a monoidal category $\CC$, we use $\Fun_\CC(\CM,\CN)$ to denote the category of $\CC$-module functors $F:\CM\to\CN$ which {\em preserve finite colimits} throughout this paper. We remind readers that a functor between abelian categories preserve finite colimits if and only if it is right exact.

\begin{defn}
Let $\CC$ be a monoidal category. We say that an object $a\in\CC$ is {\em left dual} to an object $b\in\CC$ and $b$ is {\em right dual} to $a$, if there exist morphisms $u:\one\to b\otimes a$ and $v:a\otimes b\to\one$ such that the compositions
$$a\simeq a\otimes\one \xrightarrow{\Id_a\otimes u} a\otimes b\otimes a \xrightarrow{v\otimes\Id_a} \one\otimes a \simeq a$$
$$b\simeq \one\otimes b \xrightarrow{u\otimes\Id_b} b\otimes a\otimes b \xrightarrow{\Id_b\otimes v} b\otimes\one \simeq b$$
are identity morphisms. We also denote $a=b^L$ and $b=a^R$.
We say that $\CC$ is {\em rigid}, if every object has both a left dual and a right dual.
\end{defn}

Let $A,B$ be algebras in a monoidal category $\CC$. Given a left $\CC$-module $\CM$, we use $\LMod_A(\CM)$ to denote the category of left $A$-modules in $\CM$. Given a right $\CC$-module $\CN$, we use $\RMod_B(\CN)$ to denote the category of right $B$-modules in $\CN$. We use $\BMod_{A|B}(\CC)$ to denote the category of $A$-$B$-bimodules in $\CC$. Note that $\LMod_A(\CC)$ is automatically a right $\CC$-module and that $\RMod_B(\CC)$ is a left $\CC$-module.

\begin{rem} \label{rem:lm-mod}
Let $A$ be an algebra in a rigid monoidal category $\CC$. Given a left $A$-module $x$, the action $A\otimes x\to x$ induces a morphism $x^L\otimes A\to x^L$ which endows $x^L$ with the structure of a right $A$-module. Therefore, the functor $x\mapsto x^L$ induces an equivalence $\LMod_A(\CC)\simeq\RMod_A(\CC)^\op$. In particular, $A^L$ defines a right $A$-module.
\end{rem}

\begin{defn}
Let $\CC$ be a monoidal category and $\CM$ a left $\CC$-module. Given objects $x,y\in\CM$, we define an object $[x,y]\in\CC$, if exists, by the mapping property
$$
\Hom_\CC(a,[x,y])\simeq\Hom_\CM(a\odot x,y),
$$
and refer to it as the {\em internal hom} between $x$ and $y$.
We say that $\CM$ is {\em enriched in $\CC$}, if $[x,y]$ exists for every pair of objects $x,y\in\CM$.
\end{defn}

\begin{rem} \label{rem:inhom}
It is well known that if $[x,x]$ exists for an object $x\in\CC$, then it defines an algebra in $\CC$.
If $\CC$ is rigid, we have a canonical isomorphism for $a,b\in \CC$, $x,y\in \CM$
$$a\otimes [x, y] \otimes b \simeq [b^R \otimes x, a\otimes y].$$
\end{rem}

\begin{rem} \label{rem:rigid-inhom}
Let $\CC$ be a rigid monoidal category that admits coequalizers, and let $A$ be an algebra in $\CC$. An easy computation shows that $[x,y]\simeq(x\otimes_A y^R)^L$ for $x,y\in\RMod_A(\CC)$. In particular, $[A,x]\simeq x$ and $[x,A^L]\simeq x^L$.
\end{rem}

\begin{rem} \label{rem:adjoint-al}
In the situation of Remark \ref{rem:rigid-inhom}, let $\CM=\RMod_A(\CC)$. Note that the forgetful functor 
$[A,-]:\CM\to\CC$
admits a left adjoint functor $-\otimes A$ as well as a right adjoint functor $-\otimes A^L$. In particular, we have a unit map $\Id_\CM \to [A,-]\otimes A^L$ and a counit map $[A,-]\otimes A \to \Id_\CM$.
\end{rem}

\begin{rem} \label{rem:division}
Let $A$ be an algebra in a monoidal category $\CC$. There is a canonical isomorphism of monoids $\Hom_\CC(\one,A) \simeq \Hom_{\RMod_A(\CC)}(A,A)$. This is a special case of the more general isomorphism $\Hom_\CC(x,y) \simeq \Hom_{\RMod_A(\CC)}(x\otimes A,y)$ for $x,y\in\RMod_A(\CC)$.
\end{rem}

Let $k$ be a field throughout this work. We denote by $\bk$ the symmetric monoidal category of finite-dimensional vector spaces over $k$.

\begin{defn}  \label{def:k-linear}
By a {\em finite category} over $k$ we mean a $\bk$-module $\CC$ that is equivalent to $\RMod_A(\bk)$ for some finite-dimensional $k$-algebra $A$;
we say that $\CC$ is {\em semisimple} if the algebra $A$ is semisimple.
By a {\em $k$-bilinear functor} $F:\CC\times\CD\to\CE$, where $\CC,\CD,\CE$ are finite categories over $k$, we mean that $F$ is $k$-bilinear on morphism and right exact separately in each variable.
\end{defn}

\begin{rem}
Note that Deligne's tensor product $\CM\boxtimes\CN$ for finite categories $\CM,\CN$ over $k$ is the universal finite category which is equipped with a $k$-bilinear functor $\boxtimes:\CM\times\CN\to\CM\boxtimes\CN$. If $\CM\simeq\RMod_A(\bk)$ and $\CN\simeq\RMod_B(\bk)$, then $\CM\boxtimes\CN\simeq\RMod_{A\otimes B}(\bk)$.
\end{rem}

\begin{defn}  \label{def:finite-monoidal}
A {\em finite monoidal category} over $k$ is a monoidal category $\CC$ such that $\CC$ is a finite category over $k$ and that the tensor product $\otimes:\CC\times\CC\to\CC$ is $k$-bilinear.
We say that a nonzero finite monoidal category is {\em indecomposable} if it is not the direct sum of two nonzero finite monoidal categories.
A {\em multi-fusion category} is a rigid semisimple monoidal category.
A {\em fusion category} is a multi-fusion category $\CC$ with a simple tensor unit.
\end{defn}

\begin{rem} \label{rem:fus-one}
If $\CC$ is a monoidal category, then $\otimes:\Hom_\CC(\one,\one)\times\Hom_\CC(\one,\one)\to\Hom_\CC(\one,\one)$ is a homomorphism of monoids. As a consequence, the monoid $\Hom_\CC(\one,\one)$ is commutative. In particular, if $\CC$ is a fusion category then $\Hom_\CC(\one,\one)$ is a field.
\end{rem}

\begin{rem} \label{rem:mf-decomp}
Let $\CC$ be an indecomposable multi-fusion category, and let $\one=\bigoplus_i e_i$ be the decomposition of the tensor unit in terms of simple objects. 
Then $\CC\simeq\bigoplus_{i,j}\CC_{i j}$ where $\CC_{i j}=e_i\otimes\CC\otimes e_j$. Each of $\CC_{i i}$ is a fusion category and we have 
$\FZ(\CC)\simeq\FZ(\CC_{i i})$ \cite[Theorem 2.5.1]{KZ}. In particular, $\Hom_{\FZ(\CC)}(\one,\one)$ is a subfield of $\Hom_{\CC_{i i}}(\one,\one)$.
\end{rem}

\begin{defn} \label{def:finite-module}
Let $\CC$ be a finite monoidal category over $k$. We say that a left $\CC$-module $\CM$ is {\em finite} if $\CM$ is a finite category over $k$ and the action $\odot:\CC\times\CM\to\CM$ is $k$-bilinear. We say that a nonzero finite left $\CC$-module is {\em indecomposable} if it is not the direct sum of two nonzero finite left $\CC$-modules.
The notions of a {\em finite right module} and a {\em finite bimodule} are defined similarly.
\end{defn}

\begin{rem}
The class of finite module categories behaves well under many categorical constructions. For example, if $\CM$ is a finite left module over a finite monoidal category $\CC$, then $\CM$ is enriched in $\CC$. If, in addition, $\CC$ is rigid, then $\CM\simeq\RMod_A(\CC)$ for some algebra $A$ in $\CC$. Conversely, $\RMod_A(\CC)$ is finite for any algebra $A$ in a finite monoidal category $\CC$. See \cite[Section 2.3]{KZ}.
\end{rem}

\begin{defn}
Let $\CC$ be a semisimple category over $k$ and let $k'/k$ be a finite extension. We say that the finite category $\CC\boxtimes\bk'$ over $k'$ is obtained by applying {\em base extension} on $\CC$.
\end{defn}

\begin{rem} \label{rem:sep-ext}
If $\CC$ is a multi-fusion category, then $\FZ(\CC\boxtimes\bk') \simeq \FZ(\CC)\boxtimes\bk'$ \cite[Proposition 2.4.7]{KZ}. If $A$ is an algebra in a semisimple monoidal category $\CC$, then $\RMod_{A\boxtimes k'}(\CC\boxtimes\bk') \simeq \RMod_A(\CC)\boxtimes\bk'$ and $\BMod_{A\boxtimes k'|A\boxtimes k'}(\CC\boxtimes\bk') \simeq \BMod_{A|A}(\CC)\boxtimes\bk'$ \cite[Proposition 2.3.12]{KZ}. That is, base extension is compatible with taking Drinfeld center and module category.
Moreover, $\CC\boxtimes\bk'$ remains semisimple for a semisimple category $\CC$ if $k'/k$ is a separable extension. Consequently, $\CC\boxtimes\bk'$ is a multi-fusion category over $k'$ if it is obtained by applying a separable base extension on a multi-fusion category $\CC$ over $k$.
\end{rem}

The following theorem is an easy consequence of Barr-Beck theorem \cite[Theorem 2.1.7]{KZ} (see \cite{EGNO} for a similar result).

\begin{thm} \label{thm:reconst}
Let $\CC$ be a rigid monoidal category that admits coequalizers, and let $\CM$ be a left $\CC$-module that admits coequalizers. Then $\CM\simeq\RMod_A(\CC)$ for some algebra $A$ if and only if the following conditions are satisfied.
\begin{enumerate}
\item $\CM$ is enriched in $\CC$.
\item There is an object $P\in\CM$ such that the functor $[P,-]:\CM\to\CC$ is conservative and preserves coequalizers.
\end{enumerate}
In this case, the functor $[P,-]$ induces an equivalence $\CM\simeq\RMod_{[P,P]}(\CC)$.
\end{thm}

\section{Semisimple algebras}

\begin{defn}
Let $A$ be an algebra in a semisimple monoidal category $\CC$. We say that $A$ is {\em semisimple} if $\RMod_A(\CC)$ is semisimple. We say that $A$ is {\em simple} if $\RMod_A(\CC)$ is an indecomposable semisimple left $\CC$-module. We say that $A$ is a {\em division algebra} if $A$ is a simple right $A$-module.
\end{defn}

\begin{rem}
In the special case $\CC=\bk$, an algebra $A$ in $\CC$ is simply a finite-dimensional algebra over $k$, and $A$ is a semisimple (resp. simple, division) algebra if and only if $A$ is an ordinary semisimple (resp. simple, division) algebra over $k$. Therefore, the above definition indeed generalizes corresponding notions in classical algebras. This notion of a semisimple algebra was introduced in \cite{KO}.
\end{rem}

\begin{rem}
In view of Remark \ref{rem:division}, if $A$ is a division algebra in a semisimple monoidal category, then $\Hom_\CC(\one,A)$ is an ordinary division algebra.
As pointed out in \cite{O}, it is unlikely that a division algebra in a multi-fusion category is semisimple, however we have no any counterexample.
\end{rem}

The following proposition is a well known result. It shows that there is a good supply of semisimple algebras.

\begin{prop} \label{prop:ind-mod}
Let $\CC$ be a multi-fusion category and let $\CM$ be a semisimple left $\CC$-module. Then the left $\CC$-module $\RMod_{[x,x]}(\CC)$ is equivalent to a direct summand of $\CM$ for every $x\in\CM$. In particular, $[x,x]$ is a semisimple algebra.
\end{prop}

\begin{proof}
Let $\CM'\subset\CM$ be the full subcategory formed by the direct summands of $a\odot x$, $a\in\CC$. Clearly, $\CM'$ is a left $\CC$-module. Let $\CM''\subset\CM$ be the full subcategory form by those objects $y$ such that $\Hom_\CM(a\odot x,y)\simeq0$ for all $a\in\CC$, i.e. $[x,y]\simeq0$. Then $\CM''$ is also a left $\CC$-module, because $[x,a\odot y] \simeq a\otimes[x,y] \simeq 0$ for $a\in\CC$, $y\in\CM''$. By construction, $\CM\simeq\CM'\oplus\CM''$.

Consider the functor $[x,-]:\CM'\to\CC$. It is exact because $\CM'$ is semisimple. Moreover, $[x,y]\simeq0$ only if $y\simeq0$ for $y\in\CM'$, i.e. the functor $[x,-]$ is conservative.
Applying Theorem \ref{thm:reconst}, we obtain $\CM'\simeq\RMod_{[x,x]}(\CC)$.
\end{proof}

\begin{cor} \label{cor:ind-mod}
Let $\CC$ be a multi-fusion category and let $\CM$ be an indecomposable semisimple left $\CC$-module. Then $\RMod_{[x,x]}(\CC) \simeq \CM$ for every nonzero $x\in\CM$. 
\end{cor}

\begin{lem} \label{lem:simple-matr}
Let $A$ be a semisimple algebra in a multi-fusion category $\CC$, and let $A\simeq\bigoplus_i x_i$ be the decomposition of the right $A$-module into simple ones. Define a binary relation such that $x_i\sim x_j$ if $[x_i,x_j]\not\simeq0$. We have the following assertions:
\begin{enumerate}
\item $A\simeq\bigoplus_{i,j}[x_i,x_j]$.
\item $\sim$ is an equivalence relation.
\end{enumerate}
Suppose there is a single equivalence class for the  relation $\sim$. Then we have:
\begin{enumerate}
\setcounter{enumi}{2}
\item the functor $[x_i,-]:\RMod_A(\CC)\to\RMod_{[x_i,x_i]}(\CC)$ is an equivalence.
\item $[x_i,x_i]$ is a simple division algebra.
\item $[x_i,x_j]$ is an invertible $[x_j,x_j]$-$[x_i,x_i]$-bimodule.
\item $[x_j,x_l]\otimes_{[x_j,x_j]}[x_i,x_j]\simeq[x_i,x_l]$ as $[x_l,x_l]$-$[x_i,x_i]$-bimodules.
\end{enumerate}
\end{lem}

\begin{proof}
$(1)$ follows immediate from the isomorphism $A\simeq[A,A]$.

$(2)$ Since $x_j$ is simple, $[x_i,x_j]\not\simeq0$ if and only if $x_j$ is a direct summand of $a\otimes x_i$ for some $a\in \CC$. The relation $\sim$ is clearly reflexive. If $x_i\sim x_j$, i.e. $x_j$ is a direct summand of some $a\otimes x_i$, then the embedding $x_j \to a\otimes x_i$ induces a nonzero map $a^L\otimes x_j \to x_i$ which has to be a quotient, thus $x_j\sim x_i$. This shows that the relation is symmetric. If $x_i\sim x_j$ and $x_j\sim x_l$, i.e. $x_j$ is a direct summand of some $a\otimes x_i$ and $x_l$ is a direct summand of some $b\otimes x_j$, then $x_l$ is a direct summand of $b\otimes a\otimes x_i$, i.e. $x_i\sim x_j$. This shows that the relation is transitive. Therefore, $\sim$ is an equivalence relation.


$(3)$ Since there is a single equivalence class for $\sim$, $\RMod_A(\CC)$ is an indecomposable left $\CC$-module. Applying Corollary \ref{cor:ind-mod}, we obtain $\RMod_A(\CC)\simeq\RMod_{[x_i,x_i]}(\CC)$.

$(4)$ is a consequence of $(3)$.

$(5)$ According to $(3)$, $[x_i,x_i]$ and $[x_j,x_j]$ are Morita equivalent.
Note that the inverse of the functor $[x_i,-]$ is given by $-\otimes_{[x_i,x_i]}x_i$. Therefore, the composite equivalence $\RMod_{[x_j,x_j]}(\CC) \simeq \RMod_A(\CC) \simeq \RMod_{[x_i,x_i]}(\CC)$ carries $[x_j,x_j]$ to $[x_i,x_j]$. This implies that $[x_i,x_j]$ is an invertible bimodule.

$(6)$ The composite equivalence $\RMod_{[x_l,x_l]}(\CC)\simeq\RMod_{[x_j,x_j]}(\CC)\simeq\RMod_{[x_i,x_i]}(\CC)$ carries $[x_l,x_l]$ to $[x_j,x_l]\otimes_{[x_j,x_j]}[x_i,x_j]$, while the equivalence $\RMod_{[x_l,x_l]}(\CC)\simeq\RMod_{[x_i,x_i]}(\CC)$ carries $[x_l,x_l]$ to $[x_i, x_l]$.
\end{proof}

\begin{defn}
Let $A$ be an algebra in a semisimple monoidal category $\CC$. We say that $A$ is a {\em matrix algebra} if $A$ admits a decomposition $A\simeq\bigoplus_{i,j=1}^n A_{i j}$ ($n\ge1$) in $\CC$ such that
\begin{enumerate}
\item $A_{i i}$ are simple division algebras;
\item $A_{i j}$ is an invertible $A_{i i}$-$A_{j j}$-bimodule;
\item $A_{i j}\otimes_{A_{j j}}A_{j l} \simeq A_{i l}$ as $A_{i i}$-$A_{l l}$-bimodules;
\item the multiplication of $A$ is induced by the isomorphisms from $(3)$.
\end{enumerate}
\end{defn}

\begin{rem}
In the special case $\CC=\bk$, a division algebra is unique in its Morita class, so a matrix algebra in $\CC$ coincides with an ordinary matrix algebra.
\end{rem}

\begin{rem}
If $A$ is a matrix algebra, then $A$ is Morita equivalent to each of $A_{i i}$. In fact, $A \simeq (\bigoplus_j A_{j i}) \otimes_{A_{i i}} (\bigoplus_l A_{i l})$ by definition. Moreover, the isomorphism $A\simeq A\otimes_A A$ implies that $A_{i i} \simeq (\bigoplus_l A_{i l}) \otimes_A (\bigoplus_j A_{j i})$.
\end{rem}

\begin{thm} \label{thm:semisimple}
Let $A$ be an algebra in a multi-fusion category $\CC$.
$(1)$ $A$ is a simple algebra if and only if $A$ is a matrix algebra.
$(2)$ $A$ is a semisimple algebra if and only if $A$ is a direct sum of matrix algebras.
\end{thm}

\begin{proof}
If $A$ is a simple algebra, then $A$ is matrix algebra by Lemma \ref{lem:simple-matr}. Conversely, if $A\simeq\bigoplus_{i,j=1}^n A_{i j}$ is a matrix algebra, then $A$ is Morita equivalent to each of the simple algebras $A_{i i}$ hence $A$ is simple. This proves $(1)$. $(2)$ is a consequence of $(1)$ and Lemma \ref{lem:simple-matr}.
\end{proof}

\begin{prop} \label{prop:div}
Let $A$ be a division algebra in a multi-fusion category $\CC$. Then $A\simeq A^L$ as right $A$-modules.
\end{prop}

\begin{proof}
Let $\CM=\RMod_A(\CC)$. We have $\Hom_\CM(A,\one\otimes A^L) \simeq \Hom_\CC([A,A],\mathbf1) \not\simeq 0$ (c.f. Remark \ref{rem:adjoint-al}). So, there is a nonzero morphism of right $A$-modules $f:A\to A^L$. Since $A$ is a simple right $A$-module and since $A,A^L$ have the same length as objects of $\CC$, $f$ has to be an isomorphism.
\end{proof}

\section{Separable algebras}

The following definition is a straightforward generalization and has been extensively used in the literature.

\begin{defn}
Let $A$ be an algebra in a semisimple monoidal category $\CC$. We say that $A$ is {\em separable} if the multiplication $A\otimes A\to A$ splits as an $A$-$A$-bimodule map.
\end{defn}

\begin{prop} \label{prop:sep-semisim}
Let $A$ be a separable algebra in a semisimple monoidal category $\CC$. Then $\LMod_A(\CM)$ is semisimple for every semisimple left $\CC$-module $\CM$, and $\RMod_A(\CN)$ is semisimple for every semisimple right $\CC$-module $\CN$.
\end{prop}

\begin{proof}
The functor $x\mapsto A\odot x$ is left adjoint to the forgetful functor $\LMod_A(\CM)\to\CM$. Since $\Hom_{\LMod_A(\CM)}(A\odot x,-) \simeq \Hom_\CM(x,-)$ is exact, $A\odot x$ is a projective left $A$-module for $x\in\CM$. Let $\iota:A\to A\otimes A$ be an $A$-$A$-bimodule map which exhibits $A$ separable. Then for every left $A$-module $z\in\LMod_A(\CM)$, the map $z \simeq A\odot_A z \xrightarrow{\iota\odot_A\Id_z} (A\otimes A)\odot_A z \simeq A\odot z$ exhibits $z$ as a direct summand of the projective left $A$-module $A\odot z$. It follows that $\LMod_A(\CM)$ is semisimple. The proof for $\RMod_A(\CN)$ is similar.
\end{proof}

\begin{rem}
The separability of $A$ in Proposition \ref{prop:sep-semisim} is essential. For example, let $k'/k$ be an inseparable finite extension. Regard $k'$ as a semisimple algebra in $\bk$ and regard $\bk'$ as a semisimple left $\bk$-module. Then $\LMod_{k'}(\bk') \simeq \BMod_{k'|k'}(\bk)$ is not semisimple.
\end{rem}

\begin{cor}[\cite{O}] \label{cor:sep-semi}
Let $A$ be a separable algebras in a semisimple monoidal category $\CC$. Then
$A$ is semisimple.
\end{cor}

\begin{cor} \label{cor:sep-nece}
Let $A$ be a separable algebra in a semisimple monoidal category $\CC$ over $k$, and $\CM=\RMod_A(\CC)$. Suppose that $\Hom_\CC(a,a)$ is separable over $k$ for every $a\in\CC$. Then $\Hom_\CM(x,x)$ is separable over $k$ for every $x\in\CM$.
\end{cor}

\begin{proof}
Since $\Hom_\CC(a,a)$ is separable over $k$, $\Hom_\CC(a,a)\otimes_k k'$ is semisimple for any finite extension $k'/k$. Therefore, $\CC\boxtimes \bk'$ is semisimple. Note that $A\boxtimes k'$ is a separable algebra in $\CC\boxtimes \bk'$. So, $\CM\boxtimes\bk' \simeq \RMod_{A\boxtimes k'}(\CC\boxtimes\bk')$ is semisimple by Corollary \ref{cor:sep-semi}. It follows that $\Hom_\CM(x,x)\otimes_k k'$ is a semisimple algebra for any finite extension $k'/k$. Therefore, $\Hom_\CM(x,x)$ is separable over $k$.
\end{proof}

\begin{prop}[\cite{DSS}] \label{prop:sep}
Let $A$ be a semisimple algebra in a semisimple monoidal category $\CC$, and $\CM=\RMod_A(\CC)$.
The following conditions are equivalent:
\begin{enumerate}
\item The algebra $A$ is separable.
\item There is a separable algebra $B$ such that $\CM\simeq\RMod_B(\CC)$ as left $\CC$-modules.
\item For every semisimple left $\CC$-module $\CN$, the category $\Fun_\CC(\CM,\CN)$ is semisimple.
\item The category $\Fun_\CC(\CM,\CM)$ is semisimple.
\item The category $\BMod_{A|A}(\CC)$ is semisimple.
\end{enumerate}
\end{prop}

\begin{proof}
$(1)\Rightarrow(2)$ is obvious. For every semisimple left $\CC$-module $\CN$, the functor $F\mapsto F(A)$ gives an equivalence $\Fun_\CC(\CM,\CN) \simeq \LMod_A(\CN)$ with the inverse functor given by $x\mapsto - \otimes_A x$. Applying Proposition \ref{prop:sep-semisim}, we obtain $(2)\Rightarrow(3)$. $(3)\Rightarrow(4)$ is obvious. $(4)\Rightarrow(5)$ follows from the equivalence $\BMod_{A|A}(\CC) \simeq \Fun_\CC(\CM,\CM)$. $(5)\Rightarrow(1)$ is obvious.
\end{proof}

\begin{cor} \label{cor:mor-sep}
Let $A$ be a separable algebra in a semisimple monoidal category $\CC$. Then every algebra Morita equivalent to $A$ is separable.
\end{cor}

\begin{cor} \label{cor:sep-ext}
Let $A$ be an algebra in a semisimple monoidal category $\CC$ over $k$ and let $k'/k$ be a separable finite extension. Then $A$ is a separable algebra in $\CC$ if and only if $A\boxtimes k'$ is a separable algebra in $\CC\boxtimes\bk'$.
\end{cor}

\begin{proof}
We have $\BMod_{A\boxtimes k'|A\boxtimes k'}(\CC\boxtimes\bk') \simeq \BMod_{A|A}(\CC)\boxtimes\bk'$ (c.f. Remark \ref{rem:sep-ext}). So, $\BMod_{A\boxtimes k'|A\boxtimes k'}(\CC\boxtimes\bk')$ is semisimple if and only if $\BMod_{A|A}(\CC)$ is semisimple. Then the claim follows from Proposition \ref{prop:sep}.
\end{proof}

\begin{thm} \label{thm:separable}
An algebra in a multi-fusion category $\CC$ is separable if and only if it is a direct sum of separable matrix algebras; a matrix algebra $A\simeq\bigoplus_{i,j=1}^n A_{i j}$ in $\CC$ is separable if and only if the division algebra $A_{1 1}$ is separable.
\end{thm}

\begin{proof}
Combine Theorem \ref{thm:semisimple}, Corollary \ref{cor:sep-semi} and Corollary \ref{cor:mor-sep}.
\end{proof}

\begin{thm}\label{thm:sep-fun}
Let $A$ be an algebra in a multi-fusion category $\CC$, and $\CM=\RMod_A(\CC)$. The following conditions are equivalent:
\begin{enumerate}
\item The algebra $A$ is separable.
\item The canonical morphism $v:[A,-]\otimes A\to\Id_\CM$ in $\Fun_\CC(\CM,\CM)$ splits.
\item There exists a morphism $A^L\to A$ in $\CM$ such that the induced morphism
\begin{equation} \label{eq:idm}
\Id_\CM \to [A,-]\otimes A^L \to [A,-]\otimes A \xrightarrow{v} \Id_\CM
\end{equation}
is identity.
\item There exists a morphism $g:A^L\to A$ in $\CM$ such that the induced morphism
\begin{equation} \label{eq:beta}
\beta: A \xrightarrow{m'} A\otimes A^L \xrightarrow{\Id_A\otimes g} A\otimes A \xrightarrow{m} A
\end{equation}
is an isomorphism , where $m$ is the multiplication and $m'$ is adjoint to $m$.
\end{enumerate}
\end{thm}

\begin{proof}
$(1)\Leftrightarrow(2)$ The equivalence $\Fun_\CC(\CM,\CM) \simeq \BMod_{A|A}(\CC)$ carries the morphism $v$ to the multiplication $m:A\otimes A\to A$.


$(2)\Rightarrow(3)$ The left $\CC$-module functor $G=[A,-]:\CM\to\CC$ has a left adjoint $F=-\otimes A$ and a right adjoint $F'=-\otimes A^L$.
Suppose that $\gamma:\Id_\CM \to F\circ G$ exhibits the counit map $v$ split. Then $\gamma$ is adjoint to a morphism $F'\to F$ in $\Fun_\CC(\CC,\CM)\simeq\CM$ so that $\gamma$ is decomposed as $\Id_\CM \to F'\circ G \to F\circ G$. Note that the morphism $F'\to F$ is induced by a morphism $A^L\to A$ in $\CM$.

$(3)\Rightarrow(4)$ Apply \eqref{eq:idm} on $A$.

$(4)\Rightarrow(1)$ is obvious.
\end{proof}

\begin{cor}\label{cor:div-sep}
Let $A$ be a division algebra in a multi-fusion category $\CC$. Then $A$ is separable if and only if there exist isomorphisms of right $A$-modules $f:A\to A^L$ and $g:A^L\to A$ such that the composition
$$\alpha: \one \to A\otimes A^L \xrightarrow{f\otimes g} A^L\otimes A \to \one$$
does not vanish.
\end{cor}

\begin{proof}
The morphism $\alpha$ coincides with $\one \xrightarrow{u} A \xrightarrow{\beta} A \xrightarrow{f} A^L \xrightarrow{u^L} \one$, where $u$ is the unit and $\beta$ is defined in \eqref{eq:beta}.
Suppose that $A$ is separable. Then by Theorem \ref{thm:sep-fun}, there exists an isomorphism $g:A^L\to A$ such that $\beta$ is an isomorphism. So, there is an isomorphism $f:A\to A^L$ such that $\alpha\ne0$. Conversely, suppose there exist $f,g$ such that $\alpha\ne0$. Then $\beta$ is an isomorphism. Thus $A$ is separable by Theorem \ref{thm:sep-fun}.
\end{proof}

\section{Separability of semisimple algebras}

\begin{defn}
We say that a semisimple category $\CC$ over $k$ is {\em homogeneous} if $\Hom_\CC(x,x)\simeq k$ for every simple object $x\in\CC$.
\end{defn}

\begin{rem}
If $k$ is algebraically closed, then every semisimple category over $k$ is homogeneous.
\end{rem}

\begin{thm} \label{thm:hom-semi-sep}
Let $\CC$ be a homogeneous multi-fusion category such that $\FZ(\CC)$ is semisimple, and let $A$ be a semisimple algebra in $\CC$ such that $\RMod_A(\CC)$ is homogeneous. Then $A$ is separable.
\end{thm}

\begin{proof}
Consider the coend $W = \int^{a\in\CC}a\otimes a^R$. Since $\CC$ is homogeneous, $W \simeq \bigoplus a\otimes a^R$ where the direct sum is taken over all simple objects of $\CC$. We have $b\otimes W \simeq \int^{a\in\CC}(b\otimes a)\otimes a^R \simeq \int^{a\in\CC}a\otimes (b^L\otimes a)^R \simeq W\otimes b$ for $b\in\CC$. In this way, $W$ is equipped with a half-braiding hence defines an object of $\FZ(\CC)$.

Let $\CM=\RMod_A(\CC)$ and $F = \int^{x\in\CM}[x,-]\otimes x \in \Fun_\CC(\CM,\CM)$. Since $\CM$ is homogeneous, $F \simeq \bigoplus [x,-]\otimes x$ where the direct sum is taken over all simple objects of $\CM$. We have
\begin{eqnarray*}
F &\simeq& \int^{x\in\CM}\int^{a\in\CC}\Hom_\CC(a,[x,-])\otimes a\otimes x \\
&\simeq& \int^{x\in\CM}\int^{a\in\CC}\Hom_\CM(x,a^R\otimes -)\otimes a\otimes x \\
&\simeq& \int^{a\in\CC}a\otimes a^R\otimes - \\
&\simeq& W\otimes -.
\end{eqnarray*}
Here we used the identity $\int^{x\in\CN}\Hom_\CN(x,-)\otimes x \simeq \Id_\CN$ for a homogeneous semisimple category $\CN$.

Since $\FZ(\CC)$ is semisimple, the canonical morphism $W\to\one$ in $\FZ(\CC)$ splits. Consequently, the canonical morphism $F\to\Id_{\CM}$ in $\Fun_\CC(\CM,\CM)$ splits. We may assume $A$ is a simple algebra so that $\Id_\CM$ is a simple object of $\Fun_\CC(\CM,\CM)$. Thus $[x,-]\odot x \to \Id_\CM$ splits for some simple $x\in\CM$. Replacing $A$ by $[x,x]$ if necessary, we may assume that $x=A$. Then we conclude $A$ is separable by applying Theorem \ref{thm:sep-fun}.
\end{proof}

\begin{cor} \label{cor:cen-per-sep}
Let $\CC$ be a multi-fusion category over a perfect field $k$ such that $\FZ(\CC)$ is semisimple. Then all semisimple algebras in $\CC$ are separable.
\end{cor}

\begin{proof}
Let $A$ be a semisimple algebra in $\CC$, and $\CM=\RMod_A(\CC)$. According to Corollary \ref{cor:sep-ext}, we may assume $\CC$ and $\CM$ are homogeneous by applying separable base extension (c.f. Remark \ref{rem:sep-ext}). Then apply Theorem \ref{thm:hom-semi-sep}.
\end{proof}


In what follows, we show that the issue of inseparable field extension is never occurs in a multi-fusion category, and then generalize Corollary \ref{cor:cen-per-sep} to imperfect fields.

\begin{rem}
Let $\CC$ be an indecomposable multi-fusion category over $k$ and let $K=\Hom_{\FZ(\CC)}(\one,\one)$. Then $K$ is a field (see Remark \ref{rem:mf-decomp}). Note that the tensor product of $\CC$ is $K$-bilinear, thus $\CC$ defines a multi-fusion category over $K$. By enlarging $k$ if necessary, we may simply assume $K=k$.
\end{rem}

\begin{lem} \label{lem:mf-ext1}
Let $\CC$ be a multi-fusion category. The algebra $\Hom_\CC(\one,\one)$ is separable over $\Hom_{\FZ(\CC)}(\one,\one)$.
\end{lem}

\begin{proof}
We may assume $\CC$ is a fusion category over $k$ and $\Hom_{\FZ(\CC)}(\one,\one)=k$. By applying separable base extension on $\CC$ and enlarging $k$ correspondingly (c.f. Remark \ref{rem:sep-ext}) if necessary, we may assume further that $\Hom_\CC(a,a)$ is a purely inseparable field over $k$ for every simple $a\in\CC$. Then the embeddings $\lambda_a,\rho_a: \Hom_\CC(\one,\one) \to \Hom_\CC(a,a)$ induced by the left and right actions of $\one$ coincide, because there exists at most embedding between two purely inseparable fields. Note that $\Hom_{\FZ(\CC)}(\one,\one)$ is the maximal subfield of $\Hom_\CC(\one,\one)$ on which $\lambda_a$ and $\rho_a$ agree for all $a$. Consequently, $\Hom_{\FZ(\CC)}(\one,\one) = \Hom_\CC(\one,\one)$.
\end{proof}

\begin{rem}
In general, $\Hom_\CC(\one,\one)$ is not isomorphic to $\Hom_{\FZ(\CC)}(\one,\one)$ for a fusion category $\CC$. For example, let $k'/k$ be a separable finite extension and let $\CC=\Fun_\bk(\bk',\bk')$. Then $\CC$ is a fusion category and $\Hom_\CC(\one,\one) \simeq k'$. But $\FZ(\CC)\simeq\FZ(\bk)$ because $\CC$ is a dual category to $\bk$ \cite{EO}, so $\Hom_{\FZ(\CC)}(\one,\one) \simeq k$.
\end{rem}

\begin{lem} \label{lem:mf-ext2}
Let $\CC$ be a multi-fusion category such that $\Hom_\CC(\one,\one)$ is a direct sum of $k$. Then $\Hom_\CC(a,a)$ is separable over $k$ for every $a\in\CC$.
\end{lem}

\begin{proof}
Consider the exact functor $F:\CC\to\Fun_\bk(\CC,\CC)$, $c\mapsto c\otimes-$. Let $B=\bigoplus_a\Hom_\CC(a,a)$ where the sum is taken over all simple $a\in\CC$. Then $\Fun_\bk(\CC,\CC)$ can be identified with $\BMod_{B|B}(\bk)$, so that $F(c)$ is identified with the bimodule $\bigoplus_{a,b}\Hom_\CC(a,c\otimes b)$. 

Let $a\in\CC$ be a simple object and let $K$ be the center of $J=\Hom_\CC(a,a)$. The coequalizer diagram $a\otimes K\otimes a^R \rightrightarrows a\otimes a^R \to a\otimes_K a^R$ splits due to the semisimplicity of $\CC$. Note that $F(a\otimes a^R)$ contains the $J$-$J$-bimodule $\Hom_\CC(a,a\otimes\one) \otimes_k \Hom_\CC(\one,a^R\otimes a) \simeq J\otimes_k J$ as a direct summand. It follows that the coequalizer diagram of $J$-$J$-bimodules $J\otimes_k K\otimes_k J \rightrightarrows J\otimes_k J \to J\otimes_K J$ splits. That is, $J$ is separable over $k$.
\end{proof}

\begin{prop} \label{prop:mf-ext}
Let $\CC$ be multi-fusion category. Then $\Hom_\CC(a,a)$ is separable over $\Hom_{\FZ(\CC)}(\one,\one)$ for every $a\in\CC$.
\end{prop}

\begin{proof}
We may assume $\CC$ is indecomposable and $\Hom_{\FZ(\CC)}(\one,\one)=k$. Then $\Hom_\CC(\one,\one)$ is separable over $k$ by Lemma \ref{lem:mf-ext1}. So, by applying separable base extension, we may assume $\Hom_\CC(\one,\one)$ is a direct sum of $k$. Then apply Lemma \ref{lem:mf-ext2}.
\end{proof}

\begin{thm} \label{thm:cen-nonper-sep}
Let $\CC$ be a multi-fusion category such that $\FZ(\CC)$ is semisimple. Then a semisimple algebras $A$ in $\CC$ is separable if and only if
$\Hom_{\RMod_A(\CC)}(x,x)$ is separable over $\Hom_{\FZ(\CC)}(\one,\one)$ for every $x\in\RMod_A(\CC)$.
\end{thm}

\begin{proof}
We may assume $\CC$ is indecomposable and $\Hom_{\FZ(\CC)}(\one,\one)=k$. Necessity of the theorem follows from Proposition \ref{prop:mf-ext} and Corollary \ref{cor:sep-nece}. The proof of the other direction is parallel to that of Corollary \ref{cor:cen-per-sep}, by using Proposition \ref{prop:mf-ext}.
\end{proof}

\section{The dimension of a division algebra}

\begin{defn} \label{def:quan-dim}
Let $A$ be a division algebra in a fusion category $\CC$ over $k$ such that $\Hom_\CC(\one,A)\simeq k$ (this forces $\Hom_\CC(\one,\one)\simeq k$). The {\em dimension} of $A$, denoted as $\dim A$, is the scalar defined by $\one \to A\otimes A^L \xrightarrow{f\otimes f^{-1}} A^L\otimes A \to \one$ where $f:A\to A^L$ is an isomorphism of right $A$-modules. (Such $f$ always exists due to Proposition \ref{prop:div} and $\dim A$ is independent of the choice of $f$.)
\end{defn}

\begin{rem}
The dimension of a division algebra is related to but different from the quantum dimension of an object. For example, let $\CC$ be a homogeneous fusion category. Then the dimension of the division algebra $[a,a]\simeq a\otimes a^L$ for a simple object $a\in\CC$ coincides with the squared dimension \cite{Mu1} (or squared norm \cite{ENO}) of $a$.
\end{rem}

\begin{thm} \label{thm:div-sep-dim}
Let $A$ be a division algebra in a fusion category $\CC$ over $k$ such that $\Hom_\CC(\one,A)\simeq k$. Then $A$ is separable if and only if $\dim A\ne0$.
\end{thm}

\begin{proof}
This is immediate from Corollary \ref{cor:div-sep}.
\end{proof}

\begin{exam}
In the special case $\CC=\bk$, the only division algebra $A$ satisfying $\Hom_\CC(\one,A)\simeq k$ is the trivial algebra $k$, and we have $\dim k=1$.
\end{exam}

\begin{defn}
Let $\CC$ be a homogeneous fusion category, regarded as a left $\CC\boxtimes\CC^\rev$-module. The {\em global dimension} of $\CC$, denoted as $\dim\CC$, is the dimension of the division algebra $[\one,\one]$ in the fusion category $\CC\boxtimes\CC^\rev$.
\end{defn}

\begin{rem}
Note that $[\one,\one]\simeq\bigoplus a^L\boxtimes a$ where the sum is taken over all simple objects of $\CC$. So, the global dimension defined above agrees with that in \cite{Mu1,ENO}.
\end{rem}


\begin{cor}[\cite{Mu2,ENO,BV,DSS}] \label{cor:fus-semi-dim}
Let $\CC$ be a homogeneous fusion category. Then $\FZ(\CC)$ is semisimple if and only if $\dim\CC\ne0$.
\end{cor}

\begin{proof}
We have $\CC\simeq\RMod_{[\one,\one]}(\CC\boxtimes\CC^\rev)$ by Corollary \ref{cor:ind-mod}. Therefore, $\FZ(\CC)\simeq\Fun_{\CC\boxtimes\CC^\rev}(\CC,\CC)$ is semisimple if and only if $[\one,\one]$ is separable by Proposition \ref{prop:sep}. Then apply Theorem \ref{thm:div-sep-dim}.
\end{proof}

\begin{exam}
Finite-dimensional $\Z/p\Z$-graded vector spaces over a field $k$ of characteristic $p\ne0$ form a fusion category. It has a vanishing global dimension, thus its Drinfeld center is not semisimple.
\end{exam}

\begin{rem}
The semisimplicity of $\FZ(\CC)$ in Theorem \ref{thm:hom-semi-sep} is indispensable. For example, if $\CC$ is a homogeneous fusion category with vanishing global dimension, then the simple division algebra $[\one,\one]$ in $\CC\boxtimes\CC^\rev$ is not separable. 
\end{rem}


\begin{thm} \label{thm:zero-sep}
Let $\CC$ be a multi-fusion category over a field of characteristic zero. Then all semisimple algebras in $\CC$ are separable.
\end{thm}

\begin{proof}
According to \cite[Theorem 2.3]{ENO}, if $\CC$ is a homogeneous fusion category then $\dim\CC\ge1$, consequently $\FZ(\CC)$ is semisimple by Corollary \ref{cor:fus-semi-dim}. Moreover, $\FZ(\CC)$ is also semisimple if $\CC$ is a homogeneous multi-fusion category by Remark \ref{rem:mf-decomp}. The remaining proof is parallel to that of Corollary \ref{cor:cen-per-sep}.
\end{proof}

The following corollary generalizes \cite[Theorem 2.18]{ENO} from algebraically closed fields to perfect fields.

\begin{cor} \label{cor:fun-semi}
Let $\CC$ be a multi-fusion category over a perfect field $k$. Suppose that $\FZ(\CC)$ is semisimple or $k$ is of characteristic zero. Then $\Fun_\CC(\CM,\CN)$ is semisimple for semisimple left $\CC$-modules $\CM,\CN$.
\end{cor}

\begin{proof}
Combine Proposition \ref{prop:sep} and Corollary \ref{cor:cen-per-sep}, Theorem \ref{thm:zero-sep}.
\end{proof}


\begin{thm}
Let $\CC$ be a fusion category over an algebraically closed field $k$. Suppose $\FZ(\CC)$ is semisimple or $k$ is of characteristic zero. The following conditions for a division algebra $A$ in $\CC$ are equivalent:
\begin{enumerate}
\item
$A$ is simple.
\item
$A$ is separable.
\item
$\dim A\ne0$.
\end{enumerate}
\end{thm}

\begin{proof}
$(1)\Rightarrow(2)$ is due to Theorem \ref{thm:hom-semi-sep} and Theorem \ref{thm:zero-sep}. 
$(2)\Rightarrow(1)$ is due to Corollary \ref{cor:sep-semi}. $(2)\Leftrightarrow(3)$ is due to Theorem \ref{thm:div-sep-dim}. 
\end{proof}

\end{document}